\title{Frobenius induction for algebras}
\author{
        TIBERIU COCONE\c T \\
                \emph{Babe\c s-Bolyai University }\\
\emph{Faculty of Economics and Business Administration}\\
\emph{Str. Teodor Mihali, nr.58-60 }\\
\emph{ 400591 Cluj-Napoca, Romania}\\
\emph{tiberiu.coconet@math.ubbcluj.ro}
            \and
            ANDREI MARCUS\\
        \emph{Babe\c s Bolyai University}\\
        \emph{Faculty of Mathematics and Computer Science}\\
        \emph{ Str. Kog\u alniceanu  1, 400084}\\
         \emph{Cluj-Napoca, Romania}\\
 \emph{marcus@math.ubbcluj.ro}
 \and
        CONSTANTIN-COSMIN TODEA\\
        \emph{Technical University of Cluj-Napoca}\\
        \emph{Department of Mathematics, Str. G. Baritiu 25}\\
 \emph{Cluj-Napoca 400027, Romania}\\
 \emph{Constantin.Todea@math.utcluj.ro}
}
\date{\today}

\documentclass[10pt]{article}
\usepackage{graphicx}
\usepackage{mathptmx}      

\usepackage{amssymb,amsfonts,amsmath,amscd,latexsym}
\usepackage[english]{babel}
\usepackage{geometry}
\usepackage{latexsym}
\usepackage{amssymb}
\usepackage[all]{xy}
\usepackage{amssymb}
\usepackage{latexsym}
\usepackage{graphicx}
\usepackage{srcltx}
\newcommand{\Hom}{\mathrm{Hom}}
\newcommand{\Id}{\mathrm{Id}}

\newcommand{\End}{\mathrm{End}}
\newcommand{\Ind}{\mathrm{Ind}}
\newcommand{\IndP}{\mathrm{IndP}}
\newcommand{\IndT}{\mathrm{IndT}}
\newcommand{\Aut}{\mathrm{Aut}}

\newcommand{\Mo}{\textrm{-}\mathrm{Mod}}

\newcommand{\Ker}{\operatorname{Ker}}
\newcommand{\Ima}{\mathrm{Im}}
\newcommand{\Sum}{\displaystyle\sum}
\usepackage{amsthm}
\newtheorem{thm}{Theorem}[section]
\newtheorem{theorem}[thm]{Theorem}
\newtheorem{corollary}[thm]{Corollary}
\newtheorem{lemma}[thm]{Lemma}
\newtheorem{proposition}[thm]{Proposition}
\newtheorem{remark}[thm]{Remark}
\theoremstyle{definition}
\newtheorem{definition}[thm]{Definition}

\newtheorem{example}[thm]{Example}

\begin{document}
\maketitle

\begin{abstract}
Let $B\to A$ be a homomorphism of Hopf algebras and let $C$ be an algebra. We consider the induction from $B$ to $A$ of $C$ in two cases: when $C$ is a $B$-interior algebra and when $C$ is a $B$-module algebra. Our main results establish the connection between the two inductions. The inspiration comes from finite group representation theory, and some constructions work in even more general contexts. 
\end{abstract}

\textbf{MSC}  16S40, 16S50, 16T05, 16S35, 20C05, 19A22.

\textbf{Keywords} Hopf algebra, \and Frobenius extension, \and induction, \and augmented algebra, \and Hopf module algebra, \and smash product, \and duality.
\section{Introduction}\label{intro}

In finite group representation theory a notion of Frobenius induction for algebras was introduced by Llu\'is Puig in \cite[Definition 3.3]{Pu}. If $H$ is a subgroup of the finite group $G$, $k$ is a field and $kH\to C$ is a homomorphism of algebras ($C$ is called a $kH$-interior algebra), then
\[ \Ind_H^G C:=kG\otimes_{kH}C\otimes_{kH}kG\]
is naturally a $kG$-interior algebra. This construction has many important uses, and it is strongly related to the classical Frobenius induction for modules. If $V$ is a $kH$-module, then there is a natural isomorphism 
\[\Ind_H^G(\End_k(V))\simeq \End_k(kG\otimes_{kH}V)\] of $kG$-interior algebras. Puig also introduced in \cite{PuBook} a non-injective version, so $\Ind_H^GC$ may be defined for any group homomorphism $H\to G$, and has a similar property. Linckelmann showed in \cite{LiInd} that Puig's induction may be generalized as follows. If $A$ and $B$ are $k$-algebras, $M$ is an $(A,B)$-bimodule and $C$ is a $B$-interior algebra, then, by definition
\[\Ind_M(C):=\End_{C^\mathrm{op}}(M\otimes_B C),\] which is naturally an $A$-interior algebra.

On the other hand, if $B$ is a $k$-algebra acted upon by the subgroup $G$, Turull defined in \cite{Tur} the induced $G$-algebra 
\[\Ind_H^GC:=kG\otimes_{kH}C\]
by regarding $C$ as a $kH$-module via the given $H$-action, with multiplication 
\[(g_1\otimes c_1)(g_1\otimes c_1)=\begin{cases} g_1\otimes a_1a_2, & \textrm{if } g_1=g_2 \\ 0, & \textrm{if } g_1H\neq g_2H, \end{cases}\]
and $G$-action \[{}^{g_2}(g_1\otimes c)=g_2g_1\otimes c_1,\]
for all $c_1,c_2\in C$ and $g_1,g_2\in G$.

In this paper we are concerned with the following two problems. The first is to give conditions on the $(A,B)$-bimodule $M$ such that the induced algebra $\Ind_M(C)$ can be expressed in two ways: as an endomorphism algebra, and as a tensor product. The second problem is to find the relationship between the two types of induction. 

In Section 2 we give a positive answer to the first question  when $A$ is $\beta$-Frobenius extension of $B$ as in \cite{FM} and \cite{D}. In Section 3 we generalize the surjective version of Puig's induction to the case of a homomorphism $B\to A$ of augmented algebras with some additional conditions. In Section 4 we define Turull's induction in the situation when $B$ is a Hopf subalgebra of the Hopf algebra $A$ and $C$ is a $B$-module algebra. We also define an surjective version  of Turull's induction through a homomorphism $B\to \bar B$ of Hopf algebras. 

Our main results are given in Section 5, where we start with a Hopf subalgebra $B$ of $A$, and a $B$-module algebra $C$. Then the smash product $C\# B$ is a $B$-interior algebra, so we may construct Puig's induction from $B$ to $A$ of  $C\# B$, and also the smash product between the Turull's induced algebra $\Ind_B^A C$ and $A$. We prove in Theorem~\ref{thminjpuigturull} below that, briefly speaking, induction commutes with the construction of the smash product, and this may also be regarded as a duality theorem. In fact, a particular case of Theorem~\ref{thminjpuigturull} is related to the finite dimensional versions of some results of \cite[\S 9.4]{MoBook}. Finally, Theorem~\ref{thmsurjPuTur} is the surjective counterpart of Theorem~\ref{thminjpuigturull}.

One might ask which are the applications of these constructions and results. First, we can now generalize other results regarding induction of algebras of Puig and Turull from groups to Hopf algebras. This is the objective of a possible follow-up article. Also, notice that by the methods from this article we can induce new $k$-algebras starting from a given $k$-algebra; induction is usually used in module categories. Moreover these induced algebras applied to Hopf module algebras preserves smash products.

Our notations and assumptions are standard. If $k$  a commutative ring and $A$ is a $k$-algebra, we denote by $A\Mo$,~ $A^\mathrm{op}\Mo$ the category of (unitary) left $A$-modules, respectively right $A$-modules. Homomorphisms and subalgebras of $k$-algebras are unitary. For $n$ a positive integer and a $k$-algebra $A$ we denote by $\mathcal{M}_n(A)$ the matrix algebra.
We follow \cite{MoBook} for notations and basic facts regarding Hopf algebras, and we recall in each section the needed definitions and results. 

\section{Injective induction and $\beta$-Frobenius extensions}\label{sec2}

In this section, let $k$ be a commutative ring and let $A,B$ be two $k$-algebras. A $B$-\emph{interior} $k$-algebra is a $k$-algebra $C$ for which there is a homomorphism
$\sigma:B\rightarrow C$  of $k$-algebras. In this case, we make the convention that $C$ is a $(B,B)$-bimodule through $\sigma$, that is
\[b_1\cdot c\cdot b_2=\sigma(b_1)c\sigma(b_2)\] for any $b_1,b_2\in B$ and $c\in C$. We denote this by ${}_BC_B,$ and implicitly, the action of $B$ is through $\sigma.$

If $M$ is an $(A,B)$-bimodule,  Linckelmann  defined in \cite{LiInd} the \emph{induced algebra}
\[\Ind_M(C):=\End_{C^\mathrm{op}}(M\otimes_B C),\]
which is an $A$-interior $k$-algebra with the structural homomorphism
\[A\longrightarrow \Ind_M(C)\] mapping $a\in A$ to the $C^\mathrm{op}$-endomorphism of $M\otimes_B C$ given by left multiplication with $a$ on $M\otimes_B C$. This  definition was introduced by Linckelmann in order to generalize Puig's induction, which was defined  for interior algebras given by group algebras. We recall these ideas in the following example.

\begin{example} \label{remPuiLi} Let $H$ be a subgroup of a finite group $G$, and let $C$ be a $kH$-interior $k$-algebra. Puig defined in \cite{Pu} the injective induction from $H$ to $G$ of $C$ as the $kG$-interior algebra
\[kG\otimes_{kH}C\otimes_{kH}kG,\] with the multiplication given by
$$(x_1\otimes c\otimes y_1)\cdot(x_2\otimes d\otimes y_2)=\left\{ \begin{array}{ll}x_1\otimes cy_1x_2d\otimes y_2,~~~\text{if}~y_1x_2\in H\\
0,~~~~~~~~~~~~~~~~~~~~~~~~~~~\text{if}~y_1x_2\notin H
\end{array} \right.,$$
where $x_1,y_1,x_2,y_2\in G, c,d\in C.$ The identity of this algebra is
$$\Sum_{g\in[G/H]}g\otimes1_C\otimes g^{-1},$$
where $[G/H]$ is a set of representatives of left cosets of $H$ in $G$. If we set $M=kG$ as $(kG,kH)$-bimodule, then we have an isomorphism of interior $kG$-algebras
\[\Ind_M(C)\cong kG\otimes_{kH}C\otimes_{kH}kG.\]
\end{example}

The objective of this section is to prove that the isomorphism from Example \ref{remPuiLi} is still true in the context of  a left $\beta$-Frobenius extension of
$k$-algebras $B\leq A$, where $\beta$ a $k$-algebra automorphism  of $B$. For this we recall some basic results and notations regarding left $\beta$-Frobenius extensions from \cite{FM}.

If $M\in B\Mo$, then ${}_{\beta}M$ denotes the left $\beta$-twisted $B$-module with underlying set $M$ and left action \[b\cdot m=\beta(b)m\] for any $b\in B$ and $m\in M$. One defines similarly the right $\beta$-twisted $B$-module $M_{\beta}$, and  the $\beta$-twisted $(B,B)$-bimodule ${}_{\beta}M_{\beta}$.

By \cite[Definition 1.1]{FM} (see also \cite{D}), the algebra extension $B\leq A$ is a \emph{left $\beta$-Frobenius extension} if $A$, as right $B$-module, is finitely generated and projective, and there is an isomorphism \[A\cong {}_{\beta}\Hom_B(A,B)\] of $(B,A)$-bimodules. Moreover, in this case, by \cite[Proposition 1.3]{FM}, there is a $(B,B)$-bimodule map $\varphi: A\rightarrow {}_{\beta}B$  and there are subsets
\[\{r_i\mid i\in\{1,\ldots,n\}\},\qquad \{l_i\mid i\in\{1,\ldots,n\}\}\] of $A$ (called \emph{dual bases}) such that

\begin{equation}\label{eqasumrili} a=\Sum_{i=1}^nr_i\varphi(l_ia)=\Sum_{i=1}^n(\beta^{-1}\circ\varphi)(ar_i)l_i
\end{equation}
for all $a\in A$.

We give now the main result of this section, which says that Puig's injective induction for $\beta$-Frobenius extensions and Linckelmann's generalization agree.

\begin{theorem}\label{thm22} Let $B\leq A$ be a left $\beta$-Frobenius extension of $k$-algebras, and let $C$ be a $B$-interior algebra with structural homomorphism $\sigma:B\rightarrow C$. Then  $A_{\beta}\otimes_BC\otimes_B A$ has an $A$-interior $k$-algebra such that we have an isomorphism
\[\Ind_{A_{\beta}}(C)\cong A_{\beta}\otimes_B C\otimes_B A\]
of $A$-interior $k$-algebras.
\end{theorem}

\begin{proof} 
Define the multiplication on $A_{\beta}\otimes_BC\otimes_B A$ by
$$(a_1\otimes c_1\otimes a_1')(a_2\otimes c_2\otimes a_2'):=a_1\otimes c_1(\sigma\circ\beta^{-1}\circ\varphi)(a_1'a_2)c_2\otimes a_2',$$
for $a_1,a_1',a_2,a_2'\in A, c_1,c_2\in C.$ We verify the associativity and the existence of the identity element; the other axioms are obvious. Let
$a_1,a_1',a_2,a_2',a_3,a_3'\in A$, and let $c_1,c_2,c_3\in C$; then
\begin{align*}
((a_1\otimes c_1\otimes a_1')&(a_2\otimes c_2\otimes a_2')(a_3\otimes c_3\otimes a_3')\\
&=(a_1\otimes c_1(\sigma\circ\beta^{-1}\circ\varphi)(a_1'a_2)c_2\otimes a_2')(a_3\otimes c_3\otimes a_3')\\&=a_1\otimes c_1(\sigma\circ\beta^{-1}\circ\varphi)(a_1'a_2)c_2 (\sigma\circ\beta^{-1}\circ\varphi)(a_2'a_3)c_3\otimes a_3' \\
&=(a_1\otimes c_1\otimes a_1')(a_2\otimes c_2(\sigma\circ\beta^{-1}\circ\varphi)(a_2'a_3)c_3\otimes a_3')\\
&=(a_1\otimes c_1\otimes a_1')\left((a_2\otimes c_2\otimes a_2')(a_3\otimes c_3\otimes a_3')\right).
\end{align*}
The identity element is $1_{A_{\beta}\otimes_B C\otimes_B A}=\Sum_{i=1}^n r_i\otimes 1_C\otimes l_i$, since we have
\begin{align*}
(a\otimes c\otimes a')(\Sum_{i=1}^n r_i\otimes 1_C\otimes l_i)&=\Sum_{i=1}^na\otimes c(\sigma\circ\beta^{-1}\circ\varphi)(a'r_i)1_C\otimes l_i\\
&=\Sum_{i=1}^na\otimes c(\sigma\circ\beta^{-1}\circ\varphi)(a'r_i)\sigma(1_B)\otimes l_i\\
&=\Sum_{i=1}^na\otimes c\otimes(\beta^{-1}\circ\varphi)(a'r_i)1_B l_i\\
&=a\otimes c\otimes \Sum_{i=1}^n(\beta^{-1}\circ\varphi)(a'r_i)l_i=a\otimes c\otimes a',
\end{align*}
where the last equality is true by (\ref{eqasumrili}).

The structural homomorphism of $A_{\beta}\otimes_B C\otimes_B A$ as an $A$-interior algebra is given by
\[\tau:A\rightarrow A_{\beta}\otimes_B C\otimes_B A, \qquad a\mapsto \Sum_{i=1}^nar_i\otimes 1_C\otimes l_i\]
Indeed, we have
\begin{align*}
\tau(a_1)\tau(a_2)&=\Sum_{i=1}^n\Sum_{j=1}^n(a_1r_i\otimes 1_C\otimes l_i)(a_2r_j\otimes 1_C\otimes l_j)\\
&=\Sum_{i=1}^n\Sum_{j=1}^na_1r_i\otimes \sigma(1_B)(\sigma\circ\beta^{-1}\circ\varphi)(l_ia_2r_j)\sigma(1_B)\otimes l_j\\
&=\Sum_{i=1}^n\Sum_{j=1}^na_1r_i\varphi(l_ia_2r_j)\otimes 1_C\otimes l_j\\
&=\Sum_{j=1}^n\left(a_1\Sum_{i=1}^nr_i\varphi(l_ia_2r_j)\otimes 1_C\otimes l_j\right)\\
&=\Sum_{j=1}^na_1a_2r_j\otimes 1_C\otimes l_j=\tau(a_1a_2),
\end{align*}
where the  equality in the last line holds again by (\ref{eqasumrili}).

Explicitly, the requested isomorphism is given by
\[\Psi:A_{\beta}\otimes_B C\otimes_BA\rightarrow\Ind_{A_{\beta}}(C), \qquad a\otimes c\otimes a'\mapsto\Psi(a\otimes c\otimes a')=\Psi_{a\otimes c\otimes a'},\]
where \[\Psi_{a\otimes c\otimes a'}(b\otimes d)=a\otimes c(\sigma\circ\beta^{-1}\circ\varphi)(a'b)d\]
for any $b\in A,d\in C$.

We first verify that $\Psi$ is a homomorphism of $k$-algebras; for this let $a_1,a_2,a_1',a_2',b\in A$ and $c_1,c_2,d\in C$; then we have
\begin{align*}\Psi\left((a_1\otimes c_1\otimes a_1')(a_2\otimes c_2\otimes a_2')\right) &=\Psi(a_1\otimes c_1(\sigma\circ\beta^{-1}\circ\varphi)(a_1'a_2)c_2\otimes a_2') \\
&=\Psi_{a_1\otimes c_1(\sigma\circ\beta^{-1}\circ\varphi)(a_1'a_2)c_2\otimes a_2'},
\end{align*}
where
\[\Psi_{a_1\otimes c_1(\sigma\circ\beta^{-1}\circ\varphi)(a_1'a_2)c_2\otimes a_2'}(b\otimes d)=a_1\otimes c_1(\sigma\circ\beta^{-1}\circ\varphi)(a_1'a_2)c_2 (\sigma\circ\beta^{-1}\circ\varphi)(a_2'b)d.\]
On the other hand, we have that
$$\Psi(a_1\otimes c_1\otimes a_1')\circ \Psi(a_2\otimes c_2\otimes a_2')=\Psi_{a_1\otimes c_1\otimes a_1'}\circ \Psi_{a_2\otimes c_2\otimes a_2'},$$ where
\begin{align*}
\Psi_{a_1\otimes c_1\otimes a_1'}(\Psi_{a_2\otimes c_2\otimes a_2'}(b\otimes d))&=\Psi_{a_1\otimes c_1\otimes a_1'} (a_2\otimes c_2 (\sigma\circ\beta^{-1}\circ\varphi)(a_2'b)d)\\
  &=a_1\otimes c_1(\sigma\circ\beta^{-1}\circ\varphi)(a_1'a_2)c_2(\sigma\circ\beta^{-1}\circ\varphi)(a_2'b)d.
\end{align*}

Clearly, $\Psi_{a\otimes c\otimes a'}$ is a homomorphism of $C^\mathrm{op}$-modules. Moreover, $\Psi$ is an homomorphism of $A$-interior algebras, since it is easy to verify the commutativity of the diagram
\begin{displaymath}
 \xymatrix{A\ar[r]^{\Id_A}\ar[d]^{\tau}&A\ar[d] \\ A_{\beta}\otimes_B C\otimes_B A\ar[r]^{\Psi}&\Ind_{A_{\beta}}(C) },
\end{displaymath}
where the right-hand side arrow is the structural homomorphism from Linckelmann's definition.
Its inverse is given by
$$\Psi^{-1}:\Ind_{A_{\beta}}(C)\rightarrow A_{\beta}\otimes_B C\otimes_B A, \qquad f\mapsto \Sum_{i=1}^nf(r_i\otimes 1_C)\otimes l_i.$$
Indeed, for each $i\in\{1,\ldots,n\}$ we may write
\[f(r_i\otimes 1_C)=\sum_{j\in J} m_{j,r_i}\otimes n_{j,r_i}\in A_{\beta}\otimes_BC,\]
where $m_{j,r_i}\in A,n_{j,r_i}\in C$ for any $j\in J$, where $J$ is a finite set of indices. Let $a\in A,$ and $c\in C$. Then we have
 \begin{align*}
(\Psi\circ\Psi^{-1})(f)(a\otimes c)&=\Sum_{i=1}^n\Sum_{j\in J}\Psi(m_{j,r_i}\otimes n_{j,r_i}\otimes l_i)(a\otimes c)\\
&=\Sum_{i=1}^n\Sum_{j\in J}m_{j,r_i}\otimes n_{j,r_i}(\sigma\circ\beta^{-1}\circ\varphi)(l_ia)c\\
&=\Sum_{i=1}^nf(r_i\otimes 1_C)(\sigma\circ\beta^{-1}\circ\varphi)(l_ia)c\\
&=\Sum_{i=1}^nf(r_i\otimes (\sigma\circ\beta^{-1}\circ\varphi)(l_ia)c)\\
&=f\left(\Sum_{i=1}^nr_i\varphi(l_ia)\otimes c\right)=f(a\otimes c),
\end{align*}
where the fourth equality holds since $(\sigma\circ\beta^{-1}\circ\varphi)(l_ia)c$ is in  $C$, while the sixth equality is true by (\ref{eqasumrili}).
\end{proof}

\section{Augmented algebras and the non-injective Puig induction}\label{sec2'}

The first aim of this section is to define a non-injective induction through a homomorphism of augmented algebras which generalizes Puig's non-injective induction \cite[Section 3]{PuBook} through a homomorphism of   group algebras. We will show that it coincides with Linckelmann's generalization discussed in Section \ref{sec2}, if we choose a suitable bimodule. The second aim is related to the context of Section \ref{sec4}, since Hopf algebras are augmented algebras.

Let $k$ be a commutative ring. Recall that a $k$-algebra $A$ is \emph{augmented} if there is a homomorphism \[\alpha_A:A\rightarrow k\]  of $k$-algebras; we denote it by
$(A,\alpha_A)$. In this case we can give to $k$ a structure of trivial left (right) $A$-module (and also of trivial $(A,A)$-bimodule) through $\alpha_A$. We denote these by
${}_{\alpha_A}k$, $k_{\alpha_A},$ respectively ${}_{\alpha_A}k_{\alpha_A}$.
A homomorphism $\phi:B\rightarrow A$ between two augmented $k$-algebras $(A,\alpha_A),(B,\alpha_B)$  is a homomorphism of $k$-algebras satisfying $\alpha_A\circ \phi=\alpha_B$.

\subsection{The surjective case.}

\begin{lemma}\label{lem31} Let $\phi:B\rightarrow A$ be a surjective homomorphism of augmented $k$-algebras. Assume that $K$ is a subalgebra of $B$ such that
\[\Ker \phi\leq (\Ker\alpha_B\cap K)B,\] and let $C$ be a interior $B$-algebra with structural homomorphism $\sigma:B\rightarrow C$. Then there is an isomorphism
$$A_{\phi}\otimes_B C\cong k_{\alpha_B}\otimes_K C$$
of $C^\mathrm{op}$-modules.
\end{lemma}

\begin{proof} From the isomorphism $B/\Ker\phi\cong A$ of $k$-algebras we get that $$A_{\phi}\cong B/\Ker\phi$$ as $B^\mathrm{op}$-modules. It follows that we have an isomorphism
$$A_{\phi}\otimes_B C\cong B/\Ker \phi\otimes_B C$$
of $C^\mathrm{op}$-modules. Consider the map
$$\psi:B/\Ker\phi\otimes_B C\rightarrow k_{\alpha_B}\otimes_K C, \qquad \psi(\bar{b}\otimes c)=1\otimes \sigma(b)c,,$$ for $b\in B$, $c\in C,$ and $\bar{b}=b+\Ker\phi$.
Then $\psi$ is a well-defined homomorphism of $C^\mathrm{op}$-modules with respect to choosing a representative of $\bar{b}$, since if
$\bar{b}_1=\bar{b}_2$ for some $\bar{b}_1,\bar{b}_2\in B/\Ker \phi$, then there is $m\in\Ker\alpha_B\cap K$  and $b'\in B$ such that $b_2=b_1+mb'$, thus
$$1\otimes \sigma(b_2)c=1\otimes \sigma(b_1) c+1\otimes \sigma(m)\sigma(b')c=1\otimes \sigma(b_1)c.$$
The fact that $\psi$ is an isomorphism of $C^\mathrm{op}$-modules with its inverse
$$\psi^{-1}:k_{\alpha_B}\otimes_K C\rightarrow B/\Ker\phi\otimes_B C, \qquad 1\otimes c \mapsto \bar{1}\otimes c$$
is easy to check.
\qed\end{proof}

\begin{proposition}\label{prop32} Let $\phi:B\rightarrow A$ be a surjective homomorphism of augmented $k$-algebras. Let $C$ be
a interior $B$-algebra with structural homomorphism $\sigma:B\rightarrow C$. Then the following statements hold.
\begin{itemize}
\item[{\rm a)}]If $K$ is a subalgebra of $B$ such that
\[\Ker\alpha_B\cap K\leq \Ker \phi,\] then there is a structure of $K^\mathrm{op}$-module on $k_{\alpha_B}\otimes_K C$ such
    that $(k_{\alpha_B}\otimes_K C)^K$ is a interior $A$-algebra;
\item[{\rm b)}] If  $K$ is a subalgebra of $B$ such that
$$\Ker\alpha_B\cap K\leq\Ker \phi\leq (\Ker\alpha_B\cap K)B$$ then
$$\Ind_{A_{\phi}}(C)\cong (k_{\alpha_B}\otimes_K C)^K$$
as interior $A$-algebras.
\end{itemize}
\end{proposition}

\begin{proof} a) We have that $k_{\alpha_B}\otimes_K C$ is naturally a $K^\mathrm{op}$-module with the action
$$(1\otimes c)x=1\otimes c\sigma(x),\qquad c\in C,~x\in K.$$
Since $K$ is augmented, recall that
$$(k_{\alpha_B}\otimes_K C)^K=\left\{1\otimes c\in k_{\alpha_B}\otimes_K C\mid1\otimes c\sigma(x)=1\otimes c\alpha_B(x),~~\forall x\in K \right\}.$$
We claim that $(k_{\alpha_B}\otimes_K C)^K$ is a $k$-algebra with the multiplication
$$(1\otimes c)(1\otimes d)=1\otimes cd,\qquad 1\otimes c,~1\otimes d\in (k_{\alpha_B}\otimes_K C)^K.$$
Indeed, let $1\otimes c,1\otimes d\in(k_{\alpha_B}\otimes_K C)^K$; then we have
\begin{align*}\left((1\otimes c)(1\otimes d)\right)x & =1\otimes cd\sigma(x)=(1\otimes c)(1\otimes d\sigma(x))=(1\otimes c)(1\otimes d\alpha_B(x)) \\
 &=(1\otimes c)(1\otimes d)(1\otimes \alpha_B(x)1_C)=\left((1\otimes c)(1\otimes d)\right)\alpha_B(x).
\end{align*}
Next, it is easy to verify that the above multiplication is well-defined, associative and distributive.
We define the map
$$\sigma':B/\Ker\phi\rightarrow k_{\alpha_B}\otimes_K C, \qquad \bar{b}\mapsto 1\otimes\sigma(b),$$
which is an homomorphism of $k$-algebras. We verify that
$$\Ima\sigma'\subseteq(k_{\alpha_B}\otimes_K C)^K.$$ Indeed, let $x\in K$; then we have that
$$x-\alpha_B(x)\in\Ker\alpha_B\cap K$$ hence, for any $b\in B$ we obtain that
$$bx-b\alpha_B(x)\in B(\Ker\alpha_B\cap K).$$
But since $\Ker\alpha_B\cap K\leq \Ker \phi$, we get that $\overline{bx}=\overline{b\alpha_B(x)}.$ Since $\sigma'$ is also a homomorphism of $B^\mathrm{op}$-modules, we deduce that
$\sigma'(\bar{b})x=\sigma'(\bar{b})\alpha_B(x)$, and thus
$$\sigma'(\bar{b})\in(k_{\alpha_B}\otimes_K C)^K.$$ Since $A\cong B/\Ker\phi$ as $B^\mathrm{op}$-modules we deduce that $(k_{\alpha_B}\otimes_K C)^K$ is an interior $A$-algebra through the composition of $\sigma'$ with this isomorphism.

b) The  isomorphisms
$$\Ind_{A_{\phi}}(C)\cong \End_{C^\mathrm{op}}(k_{\alpha_B}\otimes_K C)\cong\Hom_{K^\mathrm{op}}(k,k_{\alpha_B}\otimes_K C)\cong(k_{\alpha_B}\otimes_K C)^K$$
hold, where  the first isomorphism is given by Lemma \ref{lem31}.
\end{proof}

Proposition \ref{prop32} allow us to state the next definition, and then notice that Linckelmann's generalization agree with our generalization of Puig induction through homomorphisms of augmented $k$-algebras.

\begin{definition}\label{defnsurj} Let $\phi:B\rightarrow A$ be a surjective homomorphism of augmented $k$-algebras. Assume that $K$ is a subalgebra of $B$ such that
$$\Ker\alpha_B\cap K\leq\Ker \phi\leq (\Ker\alpha_B\cap K)B.$$
Let $C$ be a interior $B$-algebra with structural homomorphism $\sigma:B\rightarrow C$. The \emph{surjective induction of $C$ through} $\phi$ is the $A$-interior algebra
$$\IndP_{\phi}(C):=(k_{\alpha_B}\otimes_K C)^K.$$
\end{definition}

\begin{example}Let $k$ be a field and let $B$ be a finite dimensional Hopf $k$-algebra. Let $K$ be a normal Hopf subalgebra of $B$, and set
$$A:=B/K^+B.$$
In this case $BK^+=K^+B$. By considering the homomorphism  $$\phi:B\rightarrow A,\qquad b\mapsto \bar{b}:=b+BK^+,$$ of Hopf algebras, we are in the situation of Proposition \ref{prop32}.
\end{example}

\subsection{The general case.}

\begin{definition}\label{defngennoninj} Let $\phi:B\rightarrow A$ be a homomorphism of augmented $k$-algebras. Assume that $K$ be is subalgebra of $B$ such that \[\Ker\alpha_B\cap K\leq\Ker \phi\leq (\Ker\alpha_B\cap K)B,\] and let $C$ be
a $B$-interior algebra with structural homomorphism $\sigma:B\rightarrow C$. The
\emph{induction of $C$ through} $\phi$ is
$$\IndP_{\phi}(C):=\Ind_M(C),$$
where we denoted $M:=A_{\phi}$, regarded as an $(A,B)$-bimodule.
\end{definition}

\begin{remark} Note that if we write $\phi=i\circ\bar{\phi}$, where the map $i$ is the inclusion from $\phi(B)$ to $A$, $\bar{\phi}:B\rightarrow \phi(B),$  and $M_1:=A$ regarded as an
$(A,\phi(B))$-bimodule, then
\[\IndP_{\phi}(C)=\Ind_{M_1}(\IndP_{\bar{\phi}}(C)).\]

In particular, if in the above definition, the algebra extension $\phi(B)\leq A$ is a left $\beta$-Frobenius extension, then by Definition \ref{defnsurj} and Theorem \ref{thm22} we have an isomorphism
$$\IndP_{\phi}(C)\cong A_{\beta}\otimes_{\phi(B)}(k_{\alpha_B}\otimes_K C)^K\otimes_{\phi(B)}A$$
of interior $A$-algebras. In the case of group algebras, we deduce \cite[Example 1.4]{LiInd}.
\end{remark}

\section{Induction for Hopf module algebras}\label{sec3}

In \cite[Definition 8.1]{Tur} A. Turull defined an induction of a $H$-algebra from a subgroup $H$ to a finite group $G$. We will generalize this to the context of Hopf module algebras. In addition, we will also define a surjective version of Turull's induction.

\subsection{The injective case.}

In this subsection let $k$ be a field, let $A$ be a finite dimensional Hopf algebras, and let $B$ be a Hopf subalgebra of $A$.  The counit of $A$ is denoted by $\varepsilon$, the comultiplication is denoted by $\Delta$ and the antipode is $S$. We will use the Sweedler notation $$\Delta(a)=\Sum a_{(1)}\otimes a_{(2)}$$ for any $a\in A.$

Let $F$ be the $k$-algebra considered in \cite{Ulb}; as a $k$-vector space $F$ consists of all right $B$-linear maps $f:A\rightarrow k$, that is,
\[F=\{f\in A^\times \mid f(ab)=f(a)\varepsilon(b), \textrm{ for all } a\in A,\ b\in B\};\] the product is given by
$$(f\cdot f')(a)=\Sum f(a_{(2)})f'(a_{(1)}),\qquad f,f'\in F.$$
The algebra $F$ has identity $\varepsilon$, and it is a left $A$-module with action
$$(af)(a')=f(S(a)a'),\qquad a,a'\in A.$$
Next, let $C_1:=A\otimes_B k$, which is an $A$-module coalgebra with comultiplication
$$C_1\rightarrow C_1\otimes C_1,~~~a\otimes 1\mapsto\Sum (a_{((1)}\otimes 1)\otimes (a_{(2)}\otimes 1).$$
It is well known that \[C_1\cong A/AB^+\] as $A$-module coalgebras. Moreover, from the proof of \cite[Lemma 1.1]{Ulb} we know that \[F\cong (C_1^\times)^\mathrm{op},\] that is, $F$ is essentially the opposite  of the $k$-dual algebra of $A\otimes_B k.$

The following lemma is probably well-known, but for completeness we include here its proof.

\begin{lemma}\label{lemFAmod} With the above notations, $F$ is a left $A$-module algebra.
\end{lemma}

\begin{proof} From the above we know that $F$ is a left $A$-module. First we verify that for any $a\in A,f,f'\in F$ we have
$$a(f\cdot f')=\Sum(a_{(1)}f)\cdot(a_{(2)}f').$$
Indeed, for any $a'\in A$, we have
$$[a(f\cdot f')](a')=(f\cdot f')(S(a)a')\Sum f((S(a)a')_{(2)})f'((S(a)a')_{(1)});$$
on the other hand, we have
\begin{align*} \left[\Sum(a_{(1)}f)\cdot(a_{(2})f')\right](a')& =\Sum\Sum(a_{(1)}f)(a_{(2)}')(a_{(2)}f')(a_{(1)}') \\
   &= \Sum\Sum f(S(a_{(1)})a_{(2)}')f'(S(a_{(2)})a_{(1)}').
\end{align*}
The last sums from the above equalities are equal, since $\Delta$ is a homomorphism of $k$-algebras, and since $\Delta$ satisfies \cite[Proposition 1.5.10]{MoBook}.
\end{proof}

By using Lemma \ref{lemFAmod}, we can give the following definition.

\begin{definition}Let  $B$ be a Hopf subalgebra of the finite dimensional Hopf algebra $A$. Let $C$ be a $B$-module algebra. Then the
\emph{induction of $C$  from $B$ to $A$}  is the $A$-module algebra $$\IndT_{B}^{A}(C):=F\otimes_k C,$$ with multiplication given by
$$(f\otimes c)(f'\otimes c')=f\cdot f'\otimes cc', \qquad f,f'\in F,~ c,c'\in C,$$
and $A$-module algebra structure given by
$$a(f\otimes c)=af\otimes c,~~~~~~a\in A,~f\in F,~ c\in C.$$
\end{definition}

\begin{example} Let $H$ be a subgroup of a finite group $G$, let $B=kH$ and let $A=kG$.  In this case the product in $F$ is given by
$$(f\cdot f')(g)=f(g)f'(g),~~~g\in G.$$
Moreover, we have  an isomorphism of left $kG$-module coalgebras
$$A\otimes_B k\cong k[G/H],$$
where $[G/H]$ is a set of representatives of left cosets of $H$ in $G$, and a (non-canonical) isomorphism of $k$-spaces between $k[G/H]$ and its $k$-dual $k[G/H]^\times$. We
obtain an isomorphism of $kG$-module algebras
$$F\otimes C\cong\left(k[G/H]^*\right)^\mathrm{op}\otimes C\cong kG\otimes_{kH}C,$$
where $kG\otimes_{kH}C$ is the induced algebra of $C$, introduced by Turull \cite[Definition 8.1]{Tur}.
\end{example}

\subsection{The surjective case.}

We may define a surjective variant of Turull's induction just by taking a certain subalgebra of invariants.

\begin{definition}\label{defnturullsurj} Let $B$ be a finite dimensional Hopf algebra, and let $K$ be a normal Hopf subalgebra in $B$. Set $\overline{B}:=B/BK^+$ (recall that in this case $BK^+=K^+B$) and let $\phi:B\rightarrow \overline{B}$ be the canonical projection.

Let $C$ be a $B$-module algebra. The \emph{surjective induction of the $B$-module algebra $C$ through} $\phi$ is the $\bar{B}$-module subalgebra of $K$-invariant elements of $C$, that is,
$$\IndT_{\phi}(C):=C^K.$$
\end{definition}

\section{The connection between Puig's induction and Turull's induction}\label{sec4}

We keep the notations and assumptions of  Section \ref{sec3}, that is, $k$ is a field and  $B$ is a Hopf subalgebra of the finite dimensional Hopf algebra $A$.

If we take $B=k$ in Lemma \ref{lemFAmod}, then $F$ is identified to $(A^\times)^\mathrm{op}$ as an $A$-module algebra, where $A^\times=\Hom_k(A,k)$ is the $k$-dual of $A$. We also denote by $$A^{*}:=\Hom_B(A,B)$$ the $B$-dual of $A$, which is naturally an $(A,B)$-bimodule.

\begin{remark}\label{remAinter} Notice that  $\left(\End_B(A)\right)^\mathrm{op}$ is an interior $A$-algebra with the structural homomorphism
$$A\rightarrow \left(\End_B(A)\right)^\mathrm{op}, \qquad a\mapsto(x\mapsto xa),$$
and $ \End_{B^\mathrm{op}}\left(A^{*}\right)$ is a $A$-interior algebra with the structural homomorphism
$$A\rightarrow \End_{B^\mathrm{op}}\left(A^{*}\right), \qquad a\mapsto(\theta\mapsto a\theta), \qquad (a\theta)(x)=\theta(xa),$$
for any $a,x\in A$ and $\theta\in A^{*}$.
\end{remark}

\begin{lemma} \label{lem51}
\begin{itemize}
\item[{\rm a)}] Regarding $(A^\times)^\mathrm{op}$ as an $A$-module algebra, there is an anti-algebra isomorphism
$$\rho^\mathrm{op}:(A^\times)^\mathrm{op}\# A\rightarrow\End_k(A),~~~~~~~~\rho^\mathrm{op}(f\# a)(x)=\sum f(S(x_{(1)}))x_{(2)}a,$$  where $f\in A^\times$ and $a,x\in A$;
\item[{\rm b)}] We have that $\rho^\mathrm{op}(F\# A)\subseteq \End_B(A)$;
\item[{\rm c)}]  There is an isomorphism of $A$-interior algebras $$\left(\End_B(A)\right)^\mathrm{op}\cong \End_{B^\mathrm{op}}\left(A^{*}\right).$$
\end{itemize}
\end{lemma}

\begin{proof} a) We prove that $\rho^\mathrm{op}$ is an anti-algebra homomorphism. Indeed, let $f,f'\in A^\times$ and $a,a',x\in A$; then we have
\begin{align*}\rho^\mathrm{op}((f\# a)(f'\#a'))(x)&=\sum \sum (f\cdot (a_{(1)}f'))(S(x_{(1)}))x_{(2)}a_{(2)}a'\\
&=\sum \sum \sum f(S(x_{(1)})_{(2)}) (a_{(1)}f'))(S(x_{(1)})_{(1)})x_{(2)}a_{(2)}a'\\
&=\sum \sum \sum f(S(x_{(1)_{(1)}}) (a_{(1)}f'))(S(x_{(1)_{(2)}}))x_{(2)}a_{(2)}a'
\\
&=\sum \sum \sum f(S(x_{(1)_{(1)}}) f'(S(a_{(1)})S(x_{(1)_{(2)}}))x_{(2)}a_{(2)}a'
\end{align*}
On the other hand, we have
\begin{align*}\rho^\mathrm{op}(f'\# a')\left(\rho^\mathrm{op}(f\#a)(x)\right)&=\rho^\mathrm{op}(f'\#a')\left( \sum f(S(x_{(1)}))x_{(2)}a\right)\\
&=\sum f(S(x_{(1)})) \rho^\mathrm{op}(f'\# a')(x_{(2)}a)\\
&=\sum\sum f(S(x_{(1)}))f'(S((x_{(2)}a)_{(1)}))(x_{(2)}a)_{(2)}a'\\
&=\sum\sum\sum f(S(x_{(1)}))f'(S((x_{(2)_{(1)}}a_{(1)}))x_{(2)_{(2)}}a_{(2)}a'\\
&=\sum\sum\sum f(S(x_{(1)}))f'(S(a_{(1)})S((x_{(2)_{(1)}}))x_{(2)_{(2)}}a_{(2)}a',
\end{align*}
and the above last sums are equal by coassociativity.

Next, it is enough to prove that $\rho^\mathrm{op}$ is injective. For this we define two maps
$$\rho':(A^\times)^\mathrm{op}\# A\rightarrow \End_k(A)$$ and
$$\psi:\End_k(A)\rightarrow\End_k(A)$$
such that $\rho'=\psi\circ \rho^\mathrm{op}$ and $\rho'$ is injective, thus $\rho^\mathrm{op}$ is injective. These maps are defined as follows.
$$\rho'(f\#a)(x):=f(S(x))a,$$ and
$$\psi(\zeta)(x):=\sum S(x_{(2)})\zeta(x_{(1)}),$$
where $f\in A^\times,$ $\zeta\in\End_k(A)$ and $a,x\in A$.
The injectivity of $\rho'$ is easy to check, since it follows by the simple argument that $\rho'$ takes a $k$-basis into a $k$-basis. We verify that $\rho'=\psi\circ \rho^\mathrm{op}$. Indeed, for any $f\in A^\times$ and $a,x\in A$ we have that
\begin{align*}(\psi\circ\rho^\mathrm{op})(f\# a)(x)
&=\sum S(x_{(2)}) \rho^\mathrm{op}(f\# a)(x_{(1)})\\
&=\sum\sum S(x_{(2)}) f(S(x_{(1)_{(1)}}))x_{(1)_{(2)}}a\\
&=\sum\sum S(x_{(2)}) x_{(1)_{(2)}}f(S(x_{(1)_{(1)}}))a\\
&=\sum\sum S(x_{(2)_{(2)}}) x_{(2)_{(1)}}f(S(x_{(1)}))a\\
&=\sum \varepsilon(x_{(2)})f(S(x_{(1)}))a\\
&=\sum f(S(x_{(1)}\varepsilon(x_{(2)}))a\\
&=f(S(ax))a=\rho'(f\#a)(x).
\end{align*}

b) Let $f\in F$, $b\in B$ and $a,x\in A$; then we have
\begin{align*}\rho^\mathrm{op}(f\# a)(bx)&=\sum f(S(b_{(1)}x_{(1)}))b_{(2)}x_{(2)}a\\
&=\sum f(S(x_{(1)})S(b_{(1)}))b_{(2)}x_{(2)}a\\
&=\sum f(S(x_{(1)}))\epsilon(S(b_{(1)}))b_{(2)}x_{(2)}a\\
&=\sum f(S(x_{(1)}))\epsilon(b_{(1)})b_{(2)}x_{(2)}a\\
&=\sum bf(S(x_{(1)}))x_{(2)}a.\end{align*}

c) Consider the map
\[\Phi:\left(\End_B(A)\right)^\mathrm{op}\to \End_{B^\mathrm{op}}\left(A^{*}\right),\qquad f\mapsto f^{*},\qquad f^{*}(\theta)=\theta\circ f,\] where $\theta\in A^{*}.$ We have that $A^{*}$ is a right $B$-module satisfying
\[(\theta\cdot b)(a)=\theta(a)b,\] for any $a\in A $ and $b\in B.$ Consequently
$$f^{*}(\theta\cdot b)(a)=((\theta\cdot b)\circ f)(a)=\theta(f(a))b=(f^{*}(\theta)\cdot b)(a),$$ hence $\Phi$ is well defined. It is easy to check that $\Phi$ is an homomorphism of interior $A$-algebras, with respect to the structure of interior $A$-algebras given in Remark \ref{remAinter}.

Since we are dealing with finite dimensional Hopf algebras, we may choose a basis $$\{e_i\mid i=1,\ldots, n\},$$  of $A$ as a left $B$-module. Further, for any $i,j\in \{1,\ldots, n\}$, the left $B$-linear maps
\[f_{i,j}:A\to A, \qquad f_{i,j}(e_k)=\delta_{i,k}\cdot e_j,\] form a right $B$-module basis of $\End_B(A)$, where $\End_B(A)$ is a right $B$-module by transporting the structure of right $B$-module of $(\mathcal{M}_n(B))^\mathrm{op}$ through the isomorphism
$$\End_B(A)\cong (\mathcal{M}_n(B))^\mathrm{op};$$
here $(\mathcal{M}_n(B))^\mathrm{op}$ is a right $B$-module by multiplying each element of a matrix on the left hand side with the same element of $B$.

Dually, we have that the set $$\{e^*_i\mid i=1,\ldots, n\},$$ where
\[(e^*_i\cdot b)(e_j)=e^*_i(e_j)b,\] is a basis of $ A^{*}$  as a right $B$-module. Similarly, the maps
$$f_{i,j}^{*}: A^{*}\to  A^{*},\qquad f_{i,j}^{*}(e^*_k)=\delta_{i,k}\cdot e^*_j$$ form a basis of $\End_{B^\mathrm{op}}\left(A^{*}\right)$ as right $B$-module, the structure being obtained as above by transporting the right $B$-module structure of $(\mathcal{M}_n(B))^\mathrm{op}$ through an isomorphism
$$\End_{B^\mathrm{op}}(A^{*})\cong (\mathcal{M}_n(B))^\mathrm{op}.$$ Now, we have that
\begin{align*}
(\Phi(f_{i,j})(e^*_k))(e_m)&=(e^*_k\circ f_{i,j})(e_m)=e^*_k(f_{i,j}(e_m))\\
&=\begin{cases}e^*_k(e_j), \mbox{ if } i=m\\ 0,~~~~~~~~ \mbox{ if } i\neq m\end{cases}\\&= \begin{cases}1, \mbox{ if } k=j \mbox{ and } i=m\\ 0, \mbox{ otherwise } m\end{cases}\\
&=(\delta_{j,k}\cdot e^*_i)(e_m)=(f^*_{j,i}(e^{*}_k))(e_m),
\end{align*}
hence $\Phi $ maps a $B$-basis of $\End_B(A^{*})$ bijectively onto a $B$-basis of $\End_{B^\mathrm{op}}(A^{*})$.
\end{proof}

\begin{remark}\label{remrhoF} By Lemma \ref{lem51} i), ii) we may define the injective homomorphism of $k$-algebras
$$\rho_F^\mathrm{op}:F\# A\rightarrow (\End_B(A))^\mathrm{op}~~~~~~~~~~\rho_F^\mathrm{op}(f\#a)=\rho^\mathrm{op}(f\# a),$$
where $f\in F$, and $a\in A.$ It is easy to check that $\rho_F^\mathrm{op}$ is an homomorphism of $A$-interior algebras, and since $$\dim_k(F\#A)=\dim_k\End_B(A)=(\dim_kA)^2/\dim_kB,$$ we deduce that $\rho_F^\mathrm{op}$ is actually an isomorphism.
\end{remark}

Recall that if  $B$ is a Hopf subalgebra of a finite dimensional Hopf algebra $A$, then, by \cite[Theorem 1.7]{FM}, there is $\beta\in\Aut_k(B)$ such that $B\leq A$ is a left $\beta$-Frobenius extension. We may now state the injective version of our main result.

\begin{theorem}\label{thminjpuigturull}   Let $B$ be a Hopf subalgebra of a finite dimensional Hopf algebra $A$, and let $C$ be a $B$-module algebra. Then there is an isomorphism
  $$\IndT_{B}^{A}(C)\#A\rightarrow\Ind_{A_{\beta}}(C\#B)$$
of $A$-interior algebras.
\end{theorem}

\begin{proof} We will construct the isomorphisms of $A$-interior algebras
$$\xymatrix{(F\otimes C)\# A\ar[r]^{\Phi~~}&(\End_B(A))^\mathrm{op}\otimes C\ar[r]&\End_{B^\mathrm{op}}(A^{*})\otimes C  }$$
and
$$\xymatrix{ \End_{B^\mathrm{op}}(A^{*})\otimes C\ar[r]^{\Psi~~~~~~~}&\End_{(C\#B)^\mathrm{op}}(A^{*}\otimes_BC\#B)\ar[r]&\End_{(C\#B)^\mathrm{op}}(A_{\beta}\otimes_BC\#B)}$$
 in a sequence of steps.

\emph{Step 1.} Let $$\Phi:(F\otimes C)\#A\rightarrow \left(\End_B(A)\right)^\mathrm{op}\otimes C,\qquad (f \otimes c)\#a \mapsto \rho_F^\mathrm{op}(f\# a)\otimes c,$$
for any $f\in F,a\in A,c\in C$. By Remark \ref{remrhoF} we know that $\rho_F^\mathrm{op}$ is an isomorphism of $A$-interior algebras, hence $\Phi$ is an isomorphism of $A$-interior algebras; the structure of $A$-interior algebra of $(F\otimes C)\#A$ is given by
$$A\rightarrow (F\otimes C)\#A, \qquad a\mapsto(\varepsilon\otimes 1)\# a,$$
while the $A$-interior structure for $\left(\End_B(A)\right)^\mathrm{op}\otimes C$ is obtained by composing the structural homomorphism from Remark \ref{remAinter} with the homomorphism of $k$-algebras
$$\left(\End_B(A)\right)^\mathrm{op}\rightarrow\left(\End_B(A)\right)^\mathrm{op}\otimes C, \qquad \eta\mapsto\eta\otimes 1.$$

\emph{Step 2.} By Lemma \ref{lem51} c) we obtain the isomorphism
$$\left(\End_B(A)\right)^\mathrm{op}\otimes C\cong \End_{B^\mathrm{op}}(A^{*})\otimes C$$ of $A$-interior algebras.

\emph{Step 3.} Define the map
$$\Psi:\End_{B^\mathrm{op}}(A^{*})\otimes C\rightarrow \End_{(C\# B)^\mathrm{op}}(A^{*}\otimes_B C\# B)$$ by
$$\Psi(f^{*}\otimes c)(\theta\otimes c'\#b')=f^{*}(\theta)\otimes cc'\#b',$$
for $f^{*}\in \End_{B^\mathrm{op}}(A^{*})$, $\theta\in A^{*}$, and $b'\in B,c,c'\in C$. This is clearly a homomorphism of $A$-interior algebras.

We only need verify that the domain and the codomain have the same dimension as $k$-vector spaces; then, by using  arguments similar to those from the end of proof of Lemma \ref{lem51}, we deduce  that $\Psi$ is bijective. Recall that
$$\End_{B^\mathrm{op}}(A^{*})\cong \mathcal{M}_n(B)^\mathrm{op}\cong (\mathcal{M}_n(k)\otimes B)^\mathrm{op}$$
as $k$-algebras, where $n$ is the number of elements of a basis of $A$ as left $B$-modules. It follows that
$$\End_{B^\mathrm{op}}(A^{*})\otimes C\cong (\mathcal{M}_n(k)\otimes B)^\mathrm{op}\otimes C$$
as $k$-algebras, and in particular $$\dim_k(\End_{B^\mathrm{op}}(A^{*})\otimes C)=n^2\cdot \dim_kB\cdot \dim_kC.$$
 By Theorem \ref{thm22} we have that
 $$\dim_k\left(\End_{(C\#B)^\mathrm{op}}(A^{*}\otimes_BC\# B)\right)=\dim_k(A_{\beta}\otimes_BC\# B\otimes_B A);$$
but the last term is  equal to
 $$\dim_k((A\otimes_Bk)\otimes C\otimes B\otimes_B A)=\dim_k(F\otimes C\otimes A)=n^2\cdot \dim_kB\cdot \dim_kC,$$
where the last equality is true since $\dim_k F=\dim_k A/\dim_kB=n$ (see the remarks from the beginning of Section \ref{sec3} and \cite[Corollary 3.2.1, Theorem 3.3.1]{MoBook}).

\emph{Step 4.} Next, by \cite[Remark 1.2, b)]{FM} it follows that $B\leq A$ is also right $\beta^{-1}$-Frobenius extension, hence by \cite[Definition 1.1]{FM}, there is an isomorphism
$$A\cong \Hom_B(A,B)_{\beta^{-1}}$$
of $(A,B)$-bimodules. Clearly,  this isomorphism induces the isomorphism
$$A_{\beta}\cong A^{*}$$
of $(A,B)$-bimodules, and hence the isomorphism
$$\End_{(C\# B)^\mathrm{op}}(A^{*}\otimes_B C\# B)\cong \End_{(C\# B)^\mathrm{op}}(A_{\beta}\otimes_B C\# B)$$
of $A$-interior algebras. We compose the isomorphisms from these four steps and we are done.
\qed\end{proof}

Combining this with Theorem \ref{thm22}, we obtain the following corollary which generalizes \cite[Theorem 1]{C}.

\begin{corollary} There is an isomorphism  of $A$-interior algebras
$$(F\otimes C)\#A\rightarrow A_{\beta}\otimes_BC\#B\otimes _B A.$$
\end{corollary}

Finally, in the surjective case, we have a result analogous  to Theorem \ref{thminjpuigturull}.

\begin{theorem}\label{thmsurjPuTur} Let $B$ be a finite dimensional Hopf algebra and let $K$ be a normal Hopf subalgebra of $B$. Set $\overline{B}:=B/BK^+$ and let $\phi:B\rightarrow \overline{B}$ be the canonical projection.

Let $C$ be a $B$-module algebra. Then there is an isomorphism
  $$\IndT_{\phi}(C)\# \bar{B}\cong \IndP_{\phi}(C\# B)$$
of  $\bar{B}$-interior algebras.
\end{theorem}

\begin{proof} More explicitly, the requested isomorphism is
$$C^K\#\overline{B}\cong (k_{\varepsilon}\otimes_KC\#B)^K,$$
where the smash prooduct $C\# B$ is the $B$-interior algebra with the structural homomorphism
  $$B\rightarrow C\# B, \qquad b\mapsto1\#b.$$
Recall that the left action of $K$ on $C\#B$ is given by
\begin{equation}\label{eq1left}
x(c\#b)=\sum x_{(1)}c\#x_{(2)}b,
\end{equation}
for all $x\in K$, $c\in C$ and $b\in B$. For the proof we introduce another left action of $K$ on $C\# B$, by
\begin{equation}\label{eq2left}
x(c\#b)=c\#xb;
\end{equation}
notice that if  $c\in C^K$, then
$$x(c\#b)=\sum x_{(1)}c\#x_{(2)}b=\sum c\#\varepsilon(x_{(1)})x_{(2)}b=c\#xb.$$
Note also that the right action of $K$ on $B$ is
$$(c\#b)x=c\#bx,$$
and it comes from the $B$-interior algebra structure.

We show that there is an isomorphism of right $K$-modules
 $$\Phi:C\#\overline{B}\longrightarrow k_{\varepsilon}\otimes_K C\# B, \qquad c\#\bar{b}\mapsto 1\otimes c\#b,$$
for any $c\in C$, $b\in B$. Here $C\#\overline{B}$ is a right $K$-module with the right action
 $$(c\#\bar{b})x=S(x)c\#\bar{b}$$ for any $c\in C$, $b\in B$ and $x\in K$. It is clear that this action is well defined, since $\overline{B}$ is a trivial right $K$-module.

Indeed, $\Phi$ is a well-defined map, because if  $\bar{b'}=\bar{b}\in \bar{B}$, then $$b-b'=\sum_{i\in I} x_ib''_i,$$ where $ x_i\in K^+$, $b_i''\in B$ for any $i\in I$ where $I$ is a finite set of indices; by using the equality  (\ref{eq2left}), we get that
$$1\otimes c\# b=1\otimes c\#(b'+\sum_{i\in I}x_ib''_i)=1\otimes c\# b'+\sum_{i\in I}1\otimes c\#x_ib_i''=1\otimes c\# b'.$$
Next, $\Phi$ is a homomorphism of right $K$-modules, because we have
\begin{align*}
\Phi((c\#\bar{b})x)&=\Phi(S(x)c\#\bar{b})=\Phi(S(x)c\#\overline{(\varepsilon(S(x)))^{-1}bx})\\&=1\otimes\varepsilon(S(x)))^{-1}S(x)c\#bx\\&=\varepsilon(\varepsilon(S(x)))^{-1}S(x))\otimes c\#bx
\\&=\varepsilon(S(x)))^{-1}\varepsilon(S(x)))\otimes c\#bx=1\otimes c\#bx
\\&=(1\otimes c\#b)x=\Phi(c\#\bar{b})x.
\end{align*}
Finally, it is easy to see that the map
$$\Psi:k_{\varepsilon}\otimes_KC\#B\rightarrow C\#\overline{B}, \qquad 1\otimes c\#b\mapsto c\#\bar{b},$$
is  the inverse of $\Phi$.

Since $\Phi$ is an isomorphism of $K$-modules, it restricts to an isomorphism
$$(C\#\overline{B})^K\cong (k_{\varepsilon}\otimes_KC\#B)^K.$$
of vector spaces between the  subspaces $K$-invariants. We have the obvious isomorphism
$$ (C\#B)^K\cong C^K\#\overline{B},$$
and from this we obtain an isomorphism, still denoted by $\Phi$, of vector spaces
$$\Phi:C^K\#\overline{B}\rightarrow (k_{\varepsilon}\otimes_KC\#B)^K.$$
This isomorphism allows us to make the  identification
$$(k_{\varepsilon}\otimes_KC\#B)^K=(k_{\varepsilon}\otimes_KC^K\#B)^K,$$
but but the left actions (\ref{eq1left}) and (\ref{eq2left}) coincide on $C^K\# B$, we deduce that $\Phi$ is the isomorphism between the vector spaces from the statement of the theorem. The fact that $\Phi$ is actually an homomorphism of $\overline{B}$-interior algebras follows by a straightforward verification.
\end{proof}

In the particular case of a group $G$ acting on the $k$-algebra $C$, the smash product $C\# kG$ is just the skew group algebra $C*G$, and we immediately deduce the following result.

\begin{corollary} Let $G$ be a finite group, let $\phi:G\to \bar{G}$ be an group epimorphism, and denote $K:=\Ker(\phi)$. If  $C$ is a $G$-algebra, then there is a isomorphism of $k\bar{G}$-interior algebras
\[C^K*\bar{G}\simeq (k\otimes_{kK}C*G)^K,\] mapping
$c*\bar{g}$ to $1\otimes c*g,$ for  any $c\in C^K$ and any $\bar{g}\in \bar{G},$ where $g\in G$ such that $\phi(g)=\bar{g}.$
\end{corollary}


\end{document}